\documentclass[]{adamjoucc}

\usepackage{amssymb}
\usepackage{graphicx}
\usepackage{hyperref}

\makeatletter
\renewenvironment{frontmatter}
{\thispagestyle{plain}
}
{\vskip 20pt%
\blfootnote{\raggedright \ifnum\@authorcount=1\textit{E-mail address:}\else\textit{E-mail addresses:}\fi ~\@emails}%
}
\renewenvironment{abstract}
{
\hrule height 0.25pt
\vskip 5pt
\noindent \textbf{Abstract}
\vskip 5pt
}
{
\vskip 5pt
\noindent \textit{\small Keywords:~\@keywords}

\vskip 3pt
\noindent \textit{\small Math.\ Subj.\ Class.: \@msc}
\vskip 5pt
\hrule height 0.25pt
} 
\def\@oddrunninghead{S.\,Alimirzaei and D.\,W.\,Morris}
\def\@evenrunninghead{Colour-permuting automorphisms of complete Cayley graphs}
\makeatother

\theoremstyle{definition}
\newtheorem*{ack}{Acknowledgments}
\newtheorem{notation}[thm]{Notation}
\newtheorem{note}[thm]{Note}
\newtheorem*{note*}{Note}

\newenvironment{thmref}[1]{\begingroup\def\thethm{\ref{#1}}}{\endgroup}

\newcounter{caseholder} 
\newcounter{case}
\numberwithin{case}{caseholder}

\newenvironment{case}[1][\unskip]{\refstepcounter{case}\normalfont\bfseries\sffamily
\unskip\par\medskip \indent Case \thecase\ #1.\ \it}{\unskip\upshape}
\renewcommand{\thecase}{\arabic{case}}

\newcounter{subcase}

\numberwithin{subcase}{case}

\newcounter{step}

\numberwithin{step}{stepholder}
\renewcommand{\thestep}{\arabic{step}}

\newcommand{\1}{\mathbf{1}}
\newcommand{\caut}{\mathcal{A}}
\newcommand{\compose}{\mathbin{\circ}}

\newcommand{\edge}{\mathbin{\hbox{\vrule height 3 pt depth -2.25pt width 10pt}}}
\DeclareMathOperator{\id}{id}
\newcommand{\iso}{\cong}
\newcommand{\Z}{\mathbb{Z}}

\DeclareMathOperator{\Aut}{Aut}
\DeclareMathOperator{\Cay}{Cay}
\DeclareMathOperator{\Dic}{Dic}

\newcommand{\pref}[1]{(\ref{#1})}
\newcommand{\fullref}[2]{\ref{#1}\pref{#1-#2}}

\newcounter{saveenumi}
\newcommand{\MR}[1]{\href{https://mathscinet.ams.org/mathscinet-getitem?mr=#1}{MR#1}}

\begin{document}

\begin{frontmatter}   

\titledata{Colour-permuting automorphisms of complete Cayley graphs}
{} 

\begin{center}
\emph{To Dragan Maru\v{s}i\v{c} on his 70th birthday}
\bigskip 
\end{center}

\authordatatwo{Shirin Alimirzaei}{shirin.alimirzaei@uleth.ca}{}
{Dave Witte Morris}{dmorris@deductivepress.ca,
    https://deductivepress.ca/dmorris/}{}
{Department of Mathematics and  Computer Science, 
University of Lethbridge, 
\newline 4401 University Drive, 
Lethbridge, Alberta, T1K~3M4, Canada}


\begin{abstract}
Let $G$ be a (finite or infinite) group, and let $K_G = \Cay \bigl( G;G \smallsetminus \{\1\} \bigr)$ be the complete graph with vertex set~$G$, considered as a Cayley graph of~$G$. Being a Cayley graph, it has a natural edge-colouring by sets of the form $\{s, s^{-1}\}$ for $s \in G$. We prove that every colour-permuting automorphism of~$K_G$ is an affine map, unless $G \iso Q_8 \times B$, where $Q_8$ is the quaternion group of order~$8$, and $B$ is an abelian group, such that $b^2$ is trivial for all $b \in B$.

We also prove (without any restriction on~$G$) that every colour-permuting automorphism of $K_G$ is the composition of a group automorphism and a colour-preserving graph automorphism. This was conjectured by D.\,P.\,Byrne, M.\,J.\,Donner, and T.\,Q.\,Sibley in 2013.
\end{abstract}

\end{frontmatter}

\setcounter{section}{-1}

\section{Preliminaries}

The statements of our main results use the following standard notation and terminology.

\begin{defn}[{\cite[pp.~189--191]{HKMM-ColPres}}]
Let $S$ be an inverse-closed subset of a group~$G$. (This means that $s \in S \Rightarrow s^{-1} \in S$.)
	\begin{enumerate}
 
	\item The \emph{Cayley graph} of~$G$, with respect to~$S$, is the graph $\Cay(G;S)$ whose vertices are the elements of~$G$, and with an edge from~$g$ to~$gs$, for each $g \in G$ and $s \in S$.
 
	\item $\Cay(G;S)$ has a natural edge-colouring. Namely, each edge of the form $g \edge gs$ is coloured with the set $\{s,s^{-1}\}$. (In order to make the colouring well-defined, it is necessary to include $s^{-1}$, because $g \edge gs$ is the same as the edge $gs \edge g$, which is of the form $h \edge hs^{-1}$, with $h = gs$.)
 
	\item A function $\varphi \colon G \to G$ is said to be \emph{affine} if it is the  composition of an automorphism of~$G$ with left translation by an element of~$G$. This means $\varphi(x) = \alpha(gx)$, for some $\alpha \in \Aut G$ and $g \in G$.
 
	\item A Cayley graph $\Cay(G;S)$ is \emph{CCA} if all of its colour-preserving automorphisms are affine functions on~$G$.
(CCA is an abbreviation for the \emph{Cayley Colour Automorphism} property.)

	\item We say that $G$ is \emph{CCA} if every connected Cayley graph of~$G$ is CCA.
 
	\item An automorphism $\varphi$ of a Cayley graph $\Cay(G;S)$ is \emph{colour-permuting} if it respects the colour classes; this means that if two edges have the same colour, then their images under~$\varphi$ must also have the same colour. Hence, there is a permutation $\pi$ of~$S$, such that $\varphi(gs) \in \{ \varphi(g) \, \pi(s)^{\pm1} \}$ for all $g \in G$ and $s \in S$
	 (and $\pi(s^{-1}) = \pi(s)^{-1}$).
  
	\item $\Cay(G;S)$ is \emph{strongly CCA} if all of its colour-permuting automorphisms are affine functions.
 
	\item We say that $G$ is \emph{strongly CCA} if every connected Cayley graph of~$G$ is strongly CCA.
 
    \item $\Cay(G;S)$ is \emph{normal} if every automorphism of $\Cay(G;S)$ is an affine function (cf.~\cite[Definition~1]{Wang} and \cite[Lemma~2.1]{Godsil-full}).
    (\!\emph{Warning:} This terminology is not completely standard, because some graph theorists use ``normal'' to mean that the connection set~$S$ is a union of conjugacy classes \cite[p.~2]{NormalCayleyGraphs}.)
	\end{enumerate}
\end{defn}

\section{Introduction}

Many groups are CCA, but not strongly CCA. (The smallest example is the dihedral group of order 12 \cite[Prop.~7.3]{HKMM-ColPres}.) This means there are groups~$G$, such that every connected Cayley graph of~$G$ is CCA, but not every connected Cayley graph of~$G$ is strongly CCA. The following result shows that this is not true if we replace the universal quantifier with the existential quantifier; that is, it shows that if there exists a Cayley graph of~$G$ that is CCA, then there exists a Cayley graph of~$G$ that is strongly CCA. (Specifically, the complete graph on~$G$ is strongly CCA.)

\begin{thm} \label{equiv}
If $G$ is a \textup(finite or infinite\textup) group, then the following are equivalent:
    \begin{enumerate}
    \item \label{equiv-hasStrong}
    $G$ has a Cayley graph that is strongly CCA.
    \item \label{equiv-isStrong}
    The complete graph $\Cay \bigl( G; G \smallsetminus \{\1\} \bigr)$ is strongly CCA.
    \item \label{equiv-hasCCA}
    $G$ has a Cayley graph that is CCA.
    \item \label{equiv-isCCA}
    The complete graph $\Cay \bigl( G; G \smallsetminus \{\1\} \bigr)$ is CCA.
    \item \label{equiv-notHam}
    $G$ is \underline{not} of the form $Q_8 \times B$, where $Q_8$ is the quaternion group of order~$8$, and $B$ is an abelian group, such that $b^2$ is trivial for all $b \in B$.
    \setcounter{saveenumi}{\arabic{enumi}}
    \end{enumerate}
Furthermore, if $G$ is finite, then these are also equivalent to:
    \begin{enumerate} \setcounter{enumi}{\arabic{saveenumi}}
    \item \label{equiv-hasNormal}
    Either $G$ has a Cayley graph that is normal, or $G \iso \Z_4 \times \Z_2$.
    \end{enumerate}
\end{thm}

\goodbreak 

\begin{note} \label{equivNotes}
\leavevmode
    \begin{enumerate}
    \item \label{equivNotes-easy}
    The following implications of Theorem~\ref{equiv} are more-or-less obvious (see Lemma~\ref{easy}):
    \[
        \hbox{\pref{equiv-isStrong}} \Rightarrow
        \vbox{%
            \hbox{\pref{equiv-hasNormal}}
            \hbox to 1.3em{\hss\rotatebox{90}{$\Leftarrow$}\hss}
            \hbox{\pref{equiv-hasStrong}}%
            }
         \Rightarrow
        \hbox{\pref{equiv-hasCCA}} \Leftrightarrow
        \hbox{\pref{equiv-isCCA}} \Rightarrow
        \hbox{\pref{equiv-notHam}} \]
        
    \item \label{equivNotes-finiteCCA}
    The implication $(\ref{equiv-notHam} \Rightarrow \ref{equiv-isCCA})$ is also known (but it is not obvious). It was proved by D.\,P.\,Byrne, M.\,J.\,Donner, and T.\,Q.\,Sibley \cite[cf.\ p.~324]{ByrneDonnerSibley} for finite groups, and the general case was established by E.\,Dobson, A.\,Hujdurović, K.\,Kutnar, and J.\,Morris \cite[Thm.~2.3]{DHKM-CCAodd}. 

    \item \label{equivNotes-hasNormalPf}
    For groups that are  \emph{finite,} the implication $(\ref{equiv-notHam} \Rightarrow \ref{equiv-hasNormal})$ was established by C.\,Wang, D.\,Wang, and M.\,Xu \cite{Wang}. The proof relies on the classification of finite groups that do not have a Graphical Regular Representation (GRR), which is a difficult theorem that was proved in a series of papers by several authors (see \cite{Godsil-GRR}).
    \end{enumerate}
\end{note}

Therefore, the implication $(\ref{equiv-hasStrong} \Rightarrow \ref{equiv-isStrong})$ is the only new content of the theorem for finite groups, but $(\ref{equiv-hasCCA} \Rightarrow \ref{equiv-hasStrong})$ is new for infinite groups. Actually, we prove $(\ref{equiv-notHam} \Rightarrow \ref{equiv-isStrong})$ (see Corollary~\ref{NotStrongIff}), which contains both of these. 

We also prove the following result (see Section~\ref{PermSect}), which was conjectured by Byrne, Donner, and Sibley in \cite[cf.\ p.~332]{ByrneDonnerSibley}. It provides a direct proof of $(\ref{equiv-isCCA} \Rightarrow \ref{equiv-isStrong})$.

\begin{thm} \label{ColPermOfComplete}
If $G$ is any group, then every colour-permuting automorphism of the complete graph \/ $\Cay \bigl( G; G \smallsetminus \{\1\} \bigr)$ is the composition of a group automorphism and a colour-preserving graph automorphism.
\end{thm}

As an offshoot of our proof of Theorem~\ref{ColPermOfComplete}, we also establish the following result (see Corollary!\ref{NoA0=>strong}), which shows that placing a natural restriction on the colour-preserving automorphisms of a Cayley graph can also restrict the colour-permuting automorphisms:

\begin{prop}
If every colour-preserving automorphism of\/ $\Cay(G;S)$ is a translation, then every colour-permuting automorphism of\/ $\Cay(G;S)$ is an affine map \textup(i.e., $\Cay(G;S)$ is strongly CCA\textup).
\end{prop}

\begin{ack}
The authors thank Andrew Fiori for asking the question that inspired this research: which groups have a Cayley graph that is strongly CCA?
The first author also thanks Joy Morris for financial support provided by a grant from the Natural Science and Engineering Research Council of Canada (RGPIN-2017-04905).
\end{ack}

\section{Preliminaries }

Graphs in this paper are simple and undirected, but they may be infinite. 

\begin{notation}
\leavevmode
    \begin{enumerate}
    \item $G$ always denotes a group (which may be infinite).
    \item $\1$ is the identity element of~$G$.
    \item $S$ (and $S'$) always denotes an inverse-closed subset of~$G$.
    \end{enumerate}
\end{notation}

\begin{remark} \label{LoopsOrNot}
In the statements of our results, it does not matter whether graphs are allowed to have loops. This is because all of our results concern automorphisms of Cayley graphs, and adding a loop to every vertex of a graph does not affect its automorphism group.
\end{remark}

\subsection{Colour-permuting automorphisms}

The following result collects several elementary observations about colour-permuting automorphisms. All are trivial and/or well known.

\begin{lem} \label{elem}
\leavevmode
\begin{enumerate}

\item A permutation~$\varphi$ of~$G$ is a colour-preserving automorphism of\/ $\Cay(G;S)$ if and only if, for all $g \in G$ and $s \in S$, we have $\varphi(gs) \in \{\varphi(g) \, s^{\pm1} \}$. \cite[p.~190]{HKMM-ColPres}

\item For each $g \in G$, the left translation $L_g \colon x \mapsto gx$ is a colour-preserving automorphism of $\Cay(G;S)$.  \cite[p.~190]{HKMM-ColPres}

\item If $\alpha$ is an automorphism of the group~$G$, such that $\alpha(s) \in \{s^{\pm1}\}$ for all $s \in S$, then $\alpha$~is a colour-preserving automorphism of\/ $\Cay(G;S)$.  \cite[p.~190]{HKMM-ColPres}

\item \label{elem-strongCCA}
    Every strongly CCA Cayley graph is CCA. So every strongly CCA group is CCA. \cite[p.~191]{HKMM-ColPres}

\item \label{elem-normalstrongly}
    Every normal Cayley graph is strongly CCA. 
    \textup{(cf.~\cite[Remark~6.2]{HKMM-ColPres})}

\item \label{elem-minimalgenerating}
    Assume $S \subseteq S' \subseteq G$. If\/ $\Cay(G; S)$ is CCA, then\/ $\Cay(G; S')$ is CCA.
    \textup{(cf.~\cite[Lemma.~3.1]{fuller})}
\end{enumerate}
\end{lem}

We now prove (or give references for) the easy parts of Theorem~\ref{equiv} that are listed in Note~\fullref{equivNotes}{easy}.

\begin{lem} \label{easy}
\leavevmode
\begin{enumerate} \renewcommand{\theenumi}{\alph{enumi}}

\item \label{easy-hasstrong}
$(\ref{equiv-hasNormal} \Rightarrow \ref{equiv-hasStrong})$
If either $G$ has a Cayley graph that is normal, or $G \iso \Z_4 \times \Z_2$, then $G$ has a Cayley graph that is strongly CCA.
\textup{(see~\fullref{elem}{normalstrongly} for the first part)}

\item \label{easy-isstronghas}
$(\ref{equiv-isStrong} \Rightarrow \ref{equiv-hasStrong})$
If the complete graph $\Cay \bigl( G; G \smallsetminus \{\1\} \bigr)$ is strongly CCA, then $G$ has a Cayley graph that is strongly CCA.

\item $(\ref{equiv-hasStrong} \Rightarrow \ref{equiv-hasCCA})$
If $G$ has a Cayley graph that is strongly CCA, then $G$ has a Cayley graph that is CCA.
\textup{(see \fullref{elem}{strongCCA})}

\item \label{easy-hasCCAiffisCCA}
$(\ref{equiv-hasCCA} \Leftrightarrow \ref{equiv-isCCA})$
$G$ has a Cayley graph that is CCA if and only if the complete graph $\Cay \bigl( G; G \smallsetminus \{\1\} \bigr)$ is CCA.

\item $(\ref{equiv-isCCA} \Rightarrow \ref{equiv-notHam})$
If $G = Q_8 \times B$, where $b^2 = \1$ for every $b \in B$, then the complete graph\/ $\Cay \bigl( G; G \smallsetminus \{\1\} \bigr)$ is \underline{not} CCA.
\textup{\cite[1st~paragraph of proof of Thm.~2.3]{DHKM-CCAodd} (or see Example~\fullref{CompleteEgs}{Q8})}

\end{enumerate}
\end{lem}

\begin{proof}
References have been provided for all but \pref{easy-isstronghas} and \pref{easy-hasCCAiffisCCA}, and the fact (in~\pref{easy-hasstrong}) that $\Z_4 \times \Z_2$ has a Cayley graph that is strongly CCA. 
The statement in~\pref{easy-isstronghas} is completely trivial, so it requires no further comment.

To establish the nontrivial direction~$(\Rightarrow)$ of~\pref{easy-hasCCAiffisCCA}, note that if $\Cay(G;S)$ is CCA, then we see from Lemma~\fullref{elem}{minimalgenerating} that $\Cay(G; G)$ is CCA, because $S \subseteq G$. Then the graph obtained by removing the loop at every vertex is also CCA (see Remark~\ref{LoopsOrNot}), so $\Cay \bigl( G; G \smallsetminus \{\1\} \bigr)$ is CCA.

It is easy to see that $\Cay(\Z_4 \times \Z_2; S)$ is strongly CCA for $S = \{(\pm1, 0), (0,1)\}$.
\end{proof}

For completeness, we now establish a few additional well-known, elementary facts. 

\begin{lem} \label{phi(s^k)}
Assume $\varphi$ is a colour-permuting automorphism of $\Cay(G; S)$, $g \in G$, and $s,t \in S$. If $\varphi(gs) = \varphi(g) \, t$, then $\varphi(gs^n) = \varphi(g) \, t^n$ for all $n \in \Z$.
\end{lem}

\begin{proof}
Since $\varphi$ is colour-permuting, we see that it is an isomorphism from $\Cay \bigl( G; \{s^{\pm1}\} \bigr)$ to $\Cay \bigl( G; \{t^{\pm1}\} \bigr)$. So it must restrict to an isomorphism from the connected component of $\Cay \bigl( G; \{s^{\pm1}\} \bigr)$ that contains~$g$ to the connected component of $\Cay \bigl( G; \{t^{\pm1}\} \bigr)$ that contains~$\varphi(g)$. Since these components are cycles (except in the trivial cases where they have only one or two vertices), the restriction of~$\varphi$ to this component is determined by what it does to a single edge.
\end{proof}

\begin{lem} \label{phi(gh)}
Assume $\varphi$ is a colour-permuting automorphism of\/ $\Cay(G;S)$, such that $\varphi(\1) = \1$. Then, for all $g \in G$, $s \in S$, and $n \in \Z$, we have
    \[ \varphi(g s^n) \in \{ \varphi(g) \, \varphi(s)^{\pm n} \} . \]
\end{lem}

\begin{proof}
Fix $s \in S$. Since $\varphi$ is colour-permuting, there is some $t \in S$, such that
    \[ \text{$\varphi(gs) \in \{\varphi(g) \, t^{\pm1} \}$
    \quad for all $g \in G$.} \]
By taking $g = \1$, we see that $\varphi(s) \in \{ t^{\pm1} \}$. Therefore $\varphi(gs) \in \{\varphi(g) \, \varphi(s)^{\pm1} \}$ for all~$g$. The desired conclusion now follows from Lemma~\ref{phi(s^k)}.
\end{proof}

\begin{lem} \label{order2}
Assume $\varphi$ is a colour-permuting automorphism of\/ $\Cay(G;S)$, such that $\varphi(\1) = \1$, and let
    $H =\langle\, s \in S \mid s^2 = \1 \,\rangle$. 
Then $\varphi(g h) = \varphi(g) \, \varphi(h)$ for all $g \in G$ and $h \in H$. 
\end{lem}

\begin{proof}
Let $S_2$ be the set of all elements of order~$2$ in~$S$. Let $h \in H$, so $h = s_1\dots  s_n $ for some $ s_1, s_2, \ldots, s_n \in S_2$. 

Note that the elements of~$S_2$ are precisely the elements of~$S$ that corresponding to undirected edges in $\Cay(G;S)$. Since $\varphi$ is an automorphism (and $\varphi(\1) = \1$), this implies $\varphi(S_2) = S_2$, so $\varphi(s)^{-1} = \varphi(s)$ for all $s \in S_2$. Hence, we see from Lemma~\ref{phi(gh)} (with $n = 1$) that $\varphi(g s_i) = \varphi(g) \,  \varphi(s_i)$ for all~$i$.

So
     \begin{align*}
        \varphi(g h)
        &= \varphi(g s_1 s_2 \dots s_n) \\
        &= \varphi(g s_1 s_2 \dots s_{n-1}) \, \varphi (s_n)\\
        &=\varphi(gs_1 s_2 \dots s_{n-2}) \, \varphi (s_{n-1}) \, \varphi (s_n)\\
        &\hskip 1cm \vdots  \\
        &=\varphi(g) \, \varphi (s_1) \, \varphi (s_2) \cdots  \varphi (s_n)
        . \end{align*}
By taking $g = \1$, we see that 
    \[ \varphi(h) = \varphi (s_1) \, \varphi (s_2) \cdots \varphi (s_n) . \]
So $\varphi(g h) = \varphi(g) \, \varphi(h)$, as desired.
\end{proof}

\begin{lem} \label{fixlem}
Assume $\varphi$ is a colour-preserving automorphism of\/ $\Cay(G;S)$, and let
    $H =\langle\, s \in S \mid s^2 = \1 \,\rangle$. 
Then:
    \begin{enumerate}
        \item \label{fixlem-all}
        $\varphi(xh) = \varphi(x)h$,  for all $x \in G$ and $h \in H$.
    \item \label{fixlem-1}
    In particular, if $\varphi(\1) = \1$, then $\varphi(h)=h$ for all $h \in H$.
    \end{enumerate}
\end{lem}

\begin{proof}
This is very similar to the proof of Lemma~\ref{order2}.

Let $S_2$ be the set of all elements of order~$2$ in~$S$. Let $h \in H$, so $h = s_1\dots  s_n $ for some $ s_1, s_2, \ldots, s_n \in S_2$. For all $g \in G$, we have
    \[ \varphi(g s_i) 
    \in \{\varphi(g) \, s_i^{\pm 1}\}
    = \{\varphi(g) \, s_i\}
    , \]
so
    \begin{align*}
        \varphi(xh)
        &= \varphi(xs_1 s_2 \dots s_n) \\
        &= \varphi(xs_1 s_2 \dots s_{n-1})s_n \\
        &=\varphi(xs_1 s_2 \dots s_{n-2})s_{n-1}s_n\\
        &\hskip 1cm \vdots \\
        &=\varphi(x)s_1 s_2 \dots s_n\\
        &=\varphi(x)h
    . \qedhere \end{align*}
\end{proof}

\subsection{A few facts from group theory}

\begin{defn}[cf.\ {\cite[Definition 2.6]{HKMM-ColPres}}]
A group~$G$ is of \emph{dicyclic type} if it has an abelian subgroup~$A$ of index~$2$, an element~$z$ of order~$2$ that is in~$A$, and an element $q \notin A$, such that $q^2 = z$, and every element of~$A$ is inverted by~$q$ (in the action by conjugation). We write $G = \Dic(A, z, q)$.
\end{defn}

\begin{defn}[\cite{wiki-hamgrp}] \label{HamGrp}
A group~$G$ is \emph{hamiltonian} if it is nonabelian, and every subgroup of~$G$ is normal. It can be shown that $G$ is hamiltonian if and only if it is isomorphic to $Q_8 \times A \times B$, for some abelian groups $A$ and~$B$, such that every element of~$A$ has odd order, and every element of~$B$ has order $1$ or~$2$.
\end{defn}

\begin{lem}\label{propersub}
    If $H$ is a proper subgroup of a group $G$, then $\langle G \smallsetminus H \rangle = G$, where $G \smallsetminus H = \{\, g \in G \mid g \notin H \,\}$.
\end{lem}

\begin{proof}
Since $G \smallsetminus H$ is obviously contained in $\langle G \smallsetminus H \rangle$, it suffices to show that $H \subseteq \langle G \smallsetminus H \rangle$. Since $H$ is proper, there is some $g \in G \smallsetminus H$. Then, for any $h \in H$, we have $g^{-1} h \in G \smallsetminus H$, so
    \[ h = g \, (g^{-1}h) 
    \in (G \smallsetminus H) (G \smallsetminus H)
    \subseteq \langle G \smallsetminus H \rangle. \qedhere \]
\end{proof}

\begin{lem}[internal direct product {\cite[Theorem~9, p.~171]{DummitFoote}}] \label{InternalDirProd}
If $H$ and~$K$ are normal subgroups of~$G$, such that $HK = G$ and $H \cap K = \{\1\}$, then $G$ is isomorphic to the direct product of~$H$ and~$K$, and we write $G = H \times K$.
\end{lem}

\begin{lem}[{}{cf.~\cite[Proposition 14, p.~94]{DummitFoote}}] \label{G=HK}
Assume $H$ and~$K$ are subgroups of~$G$. If $G = HK$, then $G = KH$.
\end{lem}

\section{Colour-preserving automorphisms of complete Cayley graphs} \label{CCASect}

In this section, we recall a result that provides an explicit description of all the colour-preserving automorphisms of any complete Cayley graph $\Cay \bigl( G; G \smallsetminus \{\1\} \bigr)$ (see Theorem~\ref{ClassifThm}).  It was proved by D.\,P.\,Byrne, M.\,J.\,Donner, and T.\,Q.\,Sibley \cite[p.~324]{ByrneDonnerSibley} when $G$ is finite, and it was generalized to infinite groups by P.-H.\,Leemann and M.\,de\,la\,Salle \cite[Thm.~2]{LeemandelaSalle-MRR}. This theorem will be needed in Section~\ref{CompleteCCASect}, so, for completeness, we provide a proof.

\begin{exmp} \label{CompleteEgs}
Here are some important (known) examples of colour-preserving automorphisms of complete Cayley graphs.

	\begin{enumerate} \renewcommand{\theenumi}{\alph{enumi}}
	\item \label{CompleteEgs-abel}
	Assume $G$ is abelian, and define $\iota \colon G \to G$ by $\iota(g) = g^{-1}$. Then $\iota$ a colour-preserving automorphism of the complete graph $\Cay \bigl( G; G \smallsetminus \{\1\} \bigr)$ that fixes~$\1$. It is a group automorphism, and it is nontrivial if there exists $g \in G$, such that $g^2 \neq \1$.
	\item \label{CompleteEgs-Dic}
	Let $G = \Dic(A, z, q)$ be a (finite or infinite) group of dicyclic type, and define $\varphi \colon G \to G$ by $\varphi(q^i a) = q^{-i} a$ for $i \in \Z$ and $a \in A$. Then $\varphi$ is nontrivial colour-preserving automorphism of the complete graph $\Cay \bigl( G; G \smallsetminus \{\1\} \bigr)$ that fixes~$\1$, and it is a group automorphism.
	\item \label{CompleteEgs-Q8}
	Let $G$ be a (finite or infinite) hamiltonian $2$-group, so $G = Q_8 \times B$, for some group $B$ of exponent dividing~$2$. For every $I \subseteq \{i,j,k\}$, there is a unique colour-preserving automorphism~$\varphi_I$ of the complete graph $\Cay \bigl( G; G \smallsetminus \{\1\} \bigr)$ that fixes~$\1$, and satisfies 
		\[ \text{for $\ell \in \{i,j,k\}$, \ $\varphi_I(\ell) = \ell \Leftrightarrow \ell \in I$} . \]
	Furthermore, $\varphi_I$ is a group automorphism if and only if $|I|$ is odd.
	\end{enumerate}
There is some overlap between~\pref{CompleteEgs-Dic} and~\pref{CompleteEgs-Q8}, because any hamiltonian $2$-group is also a dicyclic group. For example, if we write $Q_8 \times B = \Dic \bigl( \langle k, B \rangle, i^2, i \bigr)$, then the automorphism described in~\pref{CompleteEgs-Dic} is equal to~$\varphi_{\{k\}}$.
\end{exmp}

\begin{proof}
All of these examples are included in the main theorem of \cite{ByrneDonnerSibley}, but the notation and general presentation there are very different from ours, so we provide brief justifications for the reader's convenience.

It is easy to see (and well known) that if $\varphi$ is a group automorphism, such that $\varphi(s) \in \{s^{\pm1}\}$ for all $s \in S$, then $\varphi$ is a colour-preserving automorphism of $\Cay(G; S)$ that fixes~$\1$.

\pref{CompleteEgs-abel} Note that $\iota$ is a group automorphism (because $G$ is abelian). Since $\iota(g) = g^{-1}$ is obviously in $\{g^{\pm1}\}$ for all $g \in G$, this implies that $\iota$ is a colour-preserving automorphism of $\Cay(G; S)$.

\pref{CompleteEgs-Dic}
For $a \in A$, we have $\varphi(a) = a$ and 
    \[ \varphi(qa) = q^{-1} a = a^{-1} q^{-1} = (qa)^{-1} . \]
Therefore $\varphi(g) \in \{g^{\pm 1}\}$ for all $g \in G$.

Hence, all that remains is to show that $\varphi$ is a group automorphism.
For $b \in A$, we have 
    \[ \varphi \bigl( (q^i a) b\bigr) 
    = \varphi \bigl( q^i (a b) \bigr) 
    = q^{-i} ab
    = \varphi(q^i a) \, \varphi(b)
    \]
and
    \[ \varphi \bigl( (q^i a) (qb)\bigr) 
    = \varphi \bigl( q^{i+1} (a^{-1} b) \bigr) 
    = q^{-i - 1} a^{-1} b
    = q^{-i} a \cdot q^{-i}b
    = \varphi(q^i a) \, \varphi(qb)
    . \]

\pref{CompleteEgs-Q8}
The uniqueness follows from Lemma~\ref{fixlem}.
So we only prove existence. 
    
    \begin{enumerate} \setcounter{enumi}{-1}
\item Let $\varphi_\emptyset = \iota$, where $\iota(x) = x^{-1}$, as in~\pref{CompleteEgs-abel}. This is not a group automorphism (because $G$ is not abelian). 

However, we can easily show that it is a colour-preserving graph automorphism (and it obviously fixes~$\1$), by using an argument that is taken from the first paragraph of the proof of \cite[Theorem 2.3]{DHKM-CCAodd}. Let $g,s \in G$. Then $\varphi_\emptyset(gs) = (gs)^{-1} = s^{-1} g^{-1}$. In $G  = Q_8 \times B$, every element either inverts or commutes with every other element, so $s^{-1} g^{-1} \in \{g^{-1} s^{\pm1}\} = \{\varphi(g) s^{\pm1}\}$.

    \item For $\ell \in \{i,j,k\}$, we may let $\varphi_{\{\ell\}}$ be the conjugation by~$\ell$. This is a group automorphism.

    \item If $|I| = 2$, we have $I = \{i,j,k\} \smallsetminus \{\ell\}$, for some $\ell \in \{i,j,k\}$. Then $\varphi_I = \varphi_{\{\ell\}} \compose \varphi_{\emptyset}$. This is not a group automorphism, because it is the composition of an automorphism with a non-automorphism.

    \item The automorphism $\varphi_{\{i,j,k\}}$ is the trivial permutation (i.e., the identity map). This is a group automorphism.
    \qedhere \end{enumerate}
\end{proof}

\begin{thm}[Byrne-Donner-Sibley \cite{ByrneDonnerSibley}, Leemann and de\,la\,Salle {\cite[Thm.~2]{LeemandelaSalle-MRR}}] \label{ClassifThm}
Let $G$ be a \textup(finite or infinite\textup) group. If the complete Cayley graph\/ $\Cay \bigl( G; G \smallsetminus \{\1\} \bigr)$ has a nontrivial colour-preserving automorphism  that fixes\/~$\1$, then $G$ is either abelian or dicyclic type. Furthermore, any such colour-preserving automorphism is one of the automorphisms in Example~\ref{CompleteEgs}.
\end{thm}

\begin{proof} 
This is similar to the proof of \cite[Thm.~2.3]{DHKM-CCAodd}.

By assumption, we may let $\varphi$ be a nontrivial colour-preserving automorphism of the complete graph $\Cay \bigl( G; G \smallsetminus \{\1\} \bigr)$ that fixes~$\1$.
Let 
    \[ F = \{\, f \in G \mid \varphi(f) = f\,\}, \]
and let $F^c$ be the complement of~$F$. Note that $F$ is a proper subset of~$G$, because $\varphi$ is nontrivial, and that 
    \[ \text{$\varphi(g) = g^{-1}$ for all $g \in F^c$} \]
(because $\varphi$ is colour-preserving and fixes~$\1$).

For all $f \in F$ and $h \in G$, such that $fh \in F^c$, we have
	\begin{align*}
	\varphi(fh) &= (fh)^{-1} = h^{-1} f^{-1} \\
	\intertext{and}
	\varphi(fh) &\in \{ \varphi(f) h^{\pm1} \} = \{ f h^{\pm1} \}
	\end{align*}
However, since $fh \notin F$, we know that $\varphi(fh) \neq fh$. So we must have 
    \[ f h^{-1} = \varphi(fh) = h^{-1} f^{-1} , \]
which means that $h$ inverts~$f$. Therefore $fh$ also inverts~$f$.  Thus, we have shown:
	\begin{align} \label{ClassifThm-Fc}
    \text{every element of~$F^c$ inverts every element of~$F$} .
	\end{align}

\refstepcounter{caseholder}

\begin{case} \label{ClassifThm-invert}
Assume $\varphi(x) = x^{-1}$ for all $x \in G$.
\end{case}
For all $g,h \in G$, we have
	\begin{align} \label{ClassifThm-gh=hg}
    h^{-1} g^{-1} = (gh)^{-1} = \varphi(gh) \in \{ \varphi(g) h^{\pm1} \} = \{g^{-1} h^{\pm1} \} \subseteq g^{-1} \langle h \rangle 
    , \end{align}
so $g$ normalizes~$\langle h \rangle$. Since $g$ and~$h$ are arbitrary elements of~$G$, this implies that every cyclic subgroup of~$G$ is normal, from which it follows that every subgroup of~$G$ is normal. Therefore, $G$ is either abelian or hamiltonian. 

We may assume $G$ is hamiltonian, for otherwise $G$ and~$\varphi$ are described in Example~\fullref{CompleteEgs}{abel}. 
So $G = Q_8 \times A \times B$, for some torsion abelian group~$A$ in which every element has odd order, and some group~$B$ of exponent dividing~$2$ (see Definition~\ref{HamGrp}). 

We may assume $A$ is nontrivial, for otherwise $G$ and~$\varphi$ are described in Example~\fullref{CompleteEgs}{Q8} (with $\varphi = \iota = \varphi_\emptyset$).
Now, let $a$ be a nontrivial element of~$A$. For $g,h \in G$, the calculation in~\pref{ClassifThm-gh=hg} shows that $h^{-1} g^{-1} \in \{g^{-1} h^{\pm1} \}$. By taking inverses, we conclude that $gh \in \{h^{\pm1} g\}$. Therefore, if $gh \neq hg$, then $gh = h^{-1} g$. Letting $g = j$ and $h = ia$, we conclude that
    \[ jia 
        = gh 
        = h^{-1}g 
        = (i^{-1} a^{-1}) j
        = ji a^{-1}
        ,\]
so $a = a^{-1}$. This contradicts the fact that every element of~$A$ has odd order.

\begin{note*}
Henceforth, we may assume that Case~\ref{ClassifThm-invert} does not apply, so there is some $x \in G$, such that $\varphi(x) \neq x^{-1}$. Then $\varphi(x) = x$, so $x \in F$. Since $\varphi(x) \neq x^{-1}$, we also see that $x \neq x^{-1}$. Hence, inverting the elements of~$F$ is not the same as centralizing the elements of~$F$. So we conclude from~\pref{ClassifThm-Fc} that $F$ contains its centralizer $C_G(F)$.

For any $q \in F^c$, we know that $q^2$ centralizes~$F$ (because $q$ inverts~$F$), so $q^2 \in C_G(F) \subseteq F$. This implies that $q$ inverts~$q^2$. However, $q$ obviously centralizes~$q^2$. So it must be the case that $q^{-2} = q^2$, which means $|q^2| \in \{1,2\}$. However, we also know from Lemma~\fullref{fixlem}{1} that $|q| > 2$, since $\varphi(q) \neq q$. So we conclude that $|q^2| = 2$. Thus:
    \begin{align} \label{ClassifThmPf-order4}
    \text{$|q| = 4$ \ for all $q \in F^c$.}
    \end{align}

Furthermore, for all $g,h \in F^c$, we know that $g$ and~$h$ both invert~$F$, so $gh^{-1} \in C_G(F)$. This implies that all elements of~$F^c$ are in the same coset of $C_G(F)$:
    \[ F^c \subseteq C_G(F) \, q
    \quad \text{for all $q \in F^c$} . \]
\end{note*}

\begin{case}
Assume $F$ is a subgroup of~$G$.
\end{case}
We will show that $G$ is of dicyclic type, and $\varphi$ is the automorphism in Example~\fullref{CompleteEgs}{Dic}.

Fix some $q \in F^c$. 
    \begin{itemize}
        \item Since $F^c \subseteq C_G(F) \, q \subseteq Fq$, we see that all elements of~$F^c$ are in the same coset of~$F$. This means $|G : F| = 2$. 
    
    \item Since $q$ inverts~$F$, and conjugation by~$q$ is an automorphism, we see that inversion is an automorphism of~$F$. This implies that $F$ is abelian.

    \item Let $z = q^2$. We know from the above note that $|q| = 4$, so $|z| = 2$.
    \end{itemize}
We have now established that $G = \Dic(F, z, q)$. 

So it only remains to show that $\varphi(q^i f) = q^{-i} f$ for all $i \in \Z$ and $f \in F$. 
    \begin{itemize}
        \item If $i$ is even, then $q^i f \in F$, so $\varphi(q^i f) = q^i f$. However, since $i$ is even and $|q| = 4$, we also have $q^i = q^{-i}$.
        \item If $i$ is odd, then $q^i f \in F^c$, so 
            \[ \varphi(q^i f) = (q^i f)^{-1} = f^{-1} q^{-i} = q^{-i} f , \]
        since $q$ inverts every element of~$f$ (and $i$ is odd).
    \end{itemize}

\begin{case}
Assume $F$ is not a subgroup of~$G$.
\end{case}
For convenience, let $C = C_G(F)$ be the centralizer of~$F$. As mentioned in the above note, 
we have $C \subseteq F$, and we have $F^c \subseteq Cq$ for every $q \in F^c$.

Since $F^c \subseteq Cq$, we know that $F$ contains all of the cosets of~$C$ except~$Cq$. Also, since $F$ is not a subgroup, we know that $\langle F \rangle$ is strictly larger than~$F$, which means it contains some element of~$F^c$, and therefore contains some element of~$Cq$. Since $C \subseteq F$, this implies that $\langle F \rangle$ contains~$Cq$. It also contains~$F$, so it contains all of the other cosets of~$C$, as well. Therefore $\langle F \rangle = G$. This implies that $C$ is the centre of~$G$.

Thus, all elements of~$F^c$ are in the same coset of the centre, so
    \[ \text{the elements of~$F^c$ commute with each other} . \]
Also recall (see~\pref{ClassifThm-Fc}) that
    \[ \text{every element of~$F^c$ inverts every element of~$F$} . \]

Now, let $f_1, f_2 \in F$. If $f_1 f_2 \notin F$, then it inverts both $f_1$ and~$f_2$, which implies that $f_1$ inverts~$f_2$ and $f_2$ inverts~$f_1$. On the other hand, if $f_1 f_2 \in F$, then the conjugation action of any element of~$F^c$ respects multiplication, but inverts $f_1$, $f_2$, and~$f_1f_2$, so we must have $f_1 f_2 = f_2 f_1$. Therefore, if $f_1$ and~$f_2$ are any elements of~$F$, then 
	\[ \text{either $f_1$ and~$f_2$ commute, or they invert each other} . \]

Now consider any $f \in F$, such that $f \notin C$. Then $qf \notin qC$, so $qf \notin F^c$. This means $qf \in F$, so 
    \[ qf 
    = \varphi(qf)
    \in \{ \varphi(q) \, f^{\pm1} \}
    = \{ q^{-1} \, f^{\pm1} \}
    , \]
so $q^2 \in \{\1, f^{-2}\}$. However, we know from~\pref{ClassifThmPf-order4} that $|q| = 4$, so $q^2 \neq \1$. So we conclude that 
    \begin{align} \label{ClassifThmPf-a2=f2}
    q^2 = f^{-2}
    . \end{align}
Therefore (recalling that $q$ inverts~$f$) we have
    \[ f^{-1} q f = f^{-2} q = q^2 \, q = q^3 \in \langle q \rangle . \]
So $f$ normalizes~$\langle q \rangle$. Since this is also clearly true for the elements of~$F$ that are in~$C$, we conclude that
    \[ \text{if $f \in F$ and $q \in F^c$, then $f$ normalizes~$\langle q \rangle$.} \]

We have now shown that $g$ normalizes~$\langle h \rangle$ for all $g,h \in G$. This means that every cyclic subgroup of~$G$ is normal. It follows easily that every subgroup of~$G$ is normal. (Also, $G$ is not abelian, because elements of~$F^c$ invert elements of~$F$.) Therefore $G$ is a hamiltonian group (see Definition~\ref{HamGrp}).

We claim that $G$ is a 2-group. To see this, write $G = Q_8 \times A \times B$, as in Definition~\ref{HamGrp}. We wish to show that $A$ is trivial. Suppose not. Then, since $F$ contains all but one coset of the centre, there is some $f \in F$, such that $f \notin C$, and the $A$-component of~$f$ is nontrivial. From~\pref{ClassifThmPf-a2=f2} we know that $|f^2| = 2$, which contradicts the fact that the $A$-component of~$f^2$ is nontrivial.

All elements of the centre~$C$ of~$G$ have order~$2$. Therefore, since $\varphi(q) = q^{-1}$ (and $\varphi$ is colour-preserving), we have $\varphi(qc) = \varphi(q) c^{\pm1} = q^{-1} c = (qc)^{-1}$ for all $c \in C$. So $F^c$ is the entire coset $qC$. Hence, $F^c$ contains some $\ell \in \{i,j,k\}$. Then it is easy to see that $\varphi$ is the permutation $\varphi_I$ defined in Example~\fullref{CompleteEgs}{Q8}, where $I = \{i,j,k\} \smallsetminus \{\ell\}$.
\end{proof}

The above theorem provides another proof of the (known) implication $(\ref{equiv-notHam} \Rightarrow \ref{equiv-isCCA})$ of Theorem~\ref{equiv}, which was discussed in Note~\fullref{equivNotes}{finiteCCA} of the introduction:

\begin{cor}[Byrne et al.\ {\cite[cf.\ p.~324]{ByrneDonnerSibley}}, Dobson et al.\ {\cite[Thm.~2.3]{DHKM-CCAodd}}] \label{CCAiffNotHam}
If the complete graph\/ $\Cay \bigl( G; G \smallsetminus \{\1\} \bigr)$ is not CCA, then $G$ is a hamiltonian $2$-group.
\end{cor}

\begin{proof}
Since $\Cay \bigl( G; G \smallsetminus \{\1\} \bigr)$ is not CCA, there exists a colour-preserving automorphism $\varphi$ of this complete Cayley graph, such that $\varphi$ is not affine. After composing with a translation, we may assume $\varphi(\1) = \1$. Since $\varphi$ is not affine, it is obviously not trivial, so Theorem~\ref{ClassifThm} tells us that $G$ is abelian or dicyclic, and that $\varphi$ is one of the automorphisms in Example~\ref{CompleteEgs}. However, since $\varphi$ is not affine, it is not a group automorphism, so it is not listed in part \pref{CompleteEgs-abel} or~\pref{CompleteEgs-Dic} of that example. Hence, it must be part~\pref{CompleteEgs-Q8} that applies, so $G$ is a hamiltonian $2$-group.
\end{proof}

\section{Complete graphs that are not strongly CCA} \label{CompleteCCASect}

\begin{notation} \label{ANotation}
Let $X = \Cay(G; S)$ be a Cayley graph of~$G$.
	\begin{enumerate}
	\item $\caut(X)$ is the group of all colour-permuting automorphisms of~$X$.
	\item $\caut^0(X)$ is the group of all colour-preserving automorphisms of~$X$.
	\item $\caut_\1(X)$ and $\caut^0_\1(X)$ are the stabilizers of the vertex~$\1$ in $\caut(X)$ and $\caut^0(X)$, respectively.
	\item For $\varphi \in \caut(X)$, and $g \in G$, let $\pi^\varphi_g$ be the permutation of~$S$ that is induced by $\varphi$ at the vertex~$g$. That is, for all $s \in S$, we have
		\[ \varphi(gs) = \varphi(g) \, \pi^\varphi_g(s) . \]
    \item For $\varphi \in \caut(X)$, let $\overline{\pi}^\varphi$ be the permutation of the colours that is induced by~$\varphi$. Thus, for all $g \in G$ and $s \in S$, we have
        \[ \overline{\pi} ^\varphi (\{s^{\pm1}\}) = \{ \pi^\varphi_g(s) ^{\pm1}\} . \]
	\item For $g \in G$, let $L_g$ be the left-translation by~$g$, so $L_g(x) = gx$ for all $x \in G$. This is a colour-preserving automorphism of every Cayley graph of~$G$.
	\end{enumerate}
\end{notation}

We have the following elementary (and well-known) observation:

\begin{lem} \label{AutPerm}
For $\varphi \in \caut_\1(X)$, we have $\varphi \in \Aut G$ if and only if $\pi^\varphi_g = \pi^\varphi_h$ for all $g,h \in G$.
\end{lem}

\begin{lem} \label{MakeColPres}
Given $\varphi \in \caut_\1(X)$, $g \in G$, and $s \in S$, let
    \[ \psi = L_{g}^{-1} \, \varphi^{-1} \, L_{\varphi (g)} \, L_{\varphi (gs)}^{-1} \,  \varphi \, L_{gs} . \]
Then $\psi \in \caut^0_\1(X)$ and $\psi(s) = s$.
\end{lem}

\begin{proof}
It is clear that $\psi$ is colour-permuting, because it is a composition of colour-permuting automorphisms. Furthermore, letting $\id$ denote a trivial permutation, we have
    \begin{align*}
    \overline\pi^\psi
    &= \overline\pi^{L_{g}^{-1}} \compose \overline\pi^{\varphi^{-1}} \compose \overline\pi^{L_{\varphi (g)}} \compose \overline\pi^{L_{\varphi (gs)}^{-1}} \compose  \overline\pi^{\varphi} \compose \overline\pi^{L_{gs}}
    \\[-0.5\baselineskip]
    &= \id \compose \overline\pi^{\varphi^{-1}} \compose \id \compose \id \compose \overline\pi^{\varphi} \compose \id
        && \begin{pmatrix} \text{translations are} \\ \text{colour-preserving} \end{pmatrix}
    \\[-0.5\baselineskip]
    &= \overline\pi^{\varphi^{-1} \varphi}
    \\&= \overline\pi^{\id}
    \\&= \id
    . \end{align*}
So the colour-permuting automorphism~$\psi$ is actually colour-preserving.

Also, a straightforward calculation shows that $\psi(\1) = \1$.
Therefore $\psi \in \caut^0_\1(X)$.

Furthermore,
    \begin{align*}
    \psi(s^{-1})
    &= L_{g}^{-1} \, \varphi^{-1} \, L_{\varphi (g)} \, L_{\varphi (gs)}^{-1} \,  \varphi \, L_{gs} (s^{-1})
    \\&= L_{g}^{-1} \, \varphi^{-1} \, L_{\varphi (g)} \, L_{\varphi (gs)}^{-1} \,  \varphi(g)
    \\&= L_{g}^{-1} \, \varphi^{-1} \, L_{\varphi (g)} \bigl( \varphi (gs)^{-1} \,  \varphi(g) \bigr)
    \\&= L_{g}^{-1} \, \varphi^{-1} \, L_{\varphi (g)} \bigl( \pi^\varphi_g(s)^{-1} \bigr)   
    \\&= L_{g}^{-1} \, \varphi^{-1} \, L_{\varphi (g)} \bigl( \pi^\varphi_g(s^{-1}) \bigr)   
        && \text{(Lemma~\ref{phi(s^k)} with $g = \1$ and $n = -1$)}
    \\&= L_{g}^{-1} \, \varphi^{-1} \bigl( \varphi (g s^{-1}) \bigr)   
    \\&= L_{g}^{-1} (g s^{-1})
    \\&= s^{-1}
    . \end{align*}
Now, it follows from Lemma~\ref{phi(s^k)} (with $g = \1$ and $n = -1$) that $\psi(s) = s$.
\end{proof}

\begin{lem} \label{StrongIfNoStab}
Let $X = \Cay(G; S)$ be a Cayley graph of~$G$, and let 
	\[ S^* = \bigl\{ s \in S \mid \text{no nontrivial element of $\caut^0_\1(X)$ fixes~$s$} \bigr\} . \]
If $S^*$ generates~$G$, then $X$ is strongly CCA.
\end{lem}

\begin{proof}
Given $\varphi \in \caut_\1(X)$, we wish to show that $\varphi \in \Aut G$.
Fix some $g \in G$ and $s \in S^*$, and, as in Lemma~\ref{MakeColPres}, let 
    \[ \psi = L_{g}^{-1} \, \varphi^{-1} \, L_{\varphi (g)} \, L_{\varphi (gs)}^{-1} \,  \varphi \, L_{gs} .\]
so $\psi \in \caut^0_\1(X)$ and $\psi(s) = s$. Since $s \in S^*$, it follows from the definition of~$S^*$ that $\psi$ is the identity map. Thus, for all $t \in S$, we have
$\psi(t) = t$, so, by the definition of~$\psi$, we have
    \[ L_{\varphi(g)}^{-1} \varphi L_g(t) 
    = L_{\varphi(gs)}^{-1} \varphi L_{gs}(t) ,\]
which means
    $\pi^\varphi_g(t) = \pi^\varphi_{gs}(t)$.
Therefore 
    $\pi^\varphi_g = \pi^\varphi_{gs}$. 
Since $s$ is an arbitrary element of~$S^*$ and $S^*$ generates~$G$, repeated application of this fact implies that $\pi^\varphi_{g} = \pi^\varphi_{h}$ for all $g,h \in G$. Therefore $\varphi \in \Aut(G)$ (see Lemma~\ref{AutPerm}).
\end{proof}

\begin{defn}[{\cite[p.~8]{Wielandt}}]
A permutation group is \emph{semiregular} if no nonidentity element of the group has any fixed points.
\end{defn}

\begin{cor} \label{SemiregularIsStrong}
Let $X = \Cay(G; S)$ be a connected Cayley graph of~$G$. If the action of $\caut^0_\1(X)$ on~$S$ is semiregular, then $X$ is strongly CCA.
\end{cor}

\begin{proof}
In this case, we have $S^* = S$. Since $X$ is connected, we know that $S$ generates~$G$.
\end{proof}

\begin{cor} \label{NoA0=>strong}
Let $X = \Cay(G; S)$ be a Cayley graph of~$G$. If $\caut^0_\1(X)$ is trivial, then $X$ is strongly CCA.
\end{cor}

\begin{proof}
Any trivial permutation group is obviously semiregular, so this is almost immediate from Corollary~\ref{SemiregularIsStrong}. The only issue (which is minor) is that the corollary requires $X$ to be connected. However, it is easy to see that  if $\caut^0_\1(X)$ is trivial, then either $X$ is connected, or $X$ has only two vertices.
\end{proof}

We can now establish the implication $(\ref{equiv-notHam} \Rightarrow \ref{equiv-isStrong})$ of Theorem~\ref{equiv}.

\begin{cor} \label{NotStrongIff}
If $G$ is not a hamiltonian $2$-group, then\/ $\Cay \bigl( G; G \smallsetminus \{\1\} \bigr)$ is strongly CCA.
\end{cor}

\begin{proof}
In the notation of Lemma~\ref{StrongIfNoStab}, it suffices to show that $\langle S^* \rangle = G$.

To this end, let $\varphi$ be a nontrivial element of $\caut^0_\1(X)$.
By Theorem~\ref{ClassifThm} (and the assumption that $G$ is not a hamiltonian $2$-group), $\varphi$ must be described in~\pref{CompleteEgs-abel} or~\pref{CompleteEgs-Dic} of Example~\ref{CompleteEgs}. Since each of these parts describes a unique automorphism, this implies that $\varphi$ is the only nontrivial element of $\caut^0_\1(X)$. (Also note that $S = G \smallsetminus \{\1\}$ in the complete graph $\Cay \bigl( G; G \smallsetminus \{\1\} \bigr)$).) Therefore, we have
    \[ \text{$S^* = G \smallsetminus F$, where $F = \{\, g \in G \mid \varphi(g) = g \,\}$} . \]
Since $\varphi$ is nontrivial, we know that $F$ is a proper subset of~$G$. Also, in both~\pref{CompleteEgs-abel} and~\pref{CompleteEgs-Dic} of Example~\ref{CompleteEgs}, it is easy to see that $F$ is a subgroup of~$G$. (In~\pref{CompleteEgs-abel}, it is the subgroup $\{\, x \in G \mid x^2 = \1\,\}$; in~\pref{CompleteEgs-Dic}, it is the subgroup~$A$.) Therefore, it is immediate from Lemma~\ref{propersub} that $\langle S^* \rangle = G$. 
\end{proof}

\section{Colour-permuting automorphisms of complete Cayley graphs} \label{PermSect}

In this section, we prove Theorem~\ref{ColPermOfComplete}.
For the reader's convenience, we state the result again here:

\begin{thmref}{ColPermOfComplete}
\begin{thm} 
If $G$ is any \textup(finite or infinite\textup) group, then every colour-permuting automorphism of the complete graph\/ $\Cay \bigl( G; G \smallsetminus \{\1\} \bigr)$ is the composition of a group automorphism and a colour-preserving graph automorphism. 
\end{thm}
\end{thmref}

In other words (and in the notation of Notation~\ref{ANotation}), we prove:
	\[ \caut \bigl( \Cay \bigl( G; G \smallsetminus \{\1\} \bigr) \bigr) = \Aut(G) \cdot \caut^0 \bigl( \Cay \bigl( G; G \smallsetminus \{\1\} \bigr) \bigr) . \]
(Note that the order of the factors on the right-hand side does not matter: if $H$ and~$K$ are subgroups of~$G$, such that $G = HK$, then also $G = KH$ (see Lemma~\ref{G=HK}).

\begin{proof}[Proof of Theorem~\ref{ColPermOfComplete}]
Let $K_G = \Cay \bigl( G; G \smallsetminus \{\1\} \bigr)$, and let $\varphi$ be a colour-permuting automorphism of~$K_G$. 
We wish to show that $\varphi$ is the composition of a group automorphism and a colour-preserving graph automorphism.
By composing $\varphi$ with a translation, we may assume, without loss of generality, that $\varphi(\1) = \1$.

If $K_G$ is strongly CCA, then $\varphi$ must be a group automorphism. So the desired conclusion is true in this case.

We may now assume that $K_G$ is not strongly CCA. By the contrapositive of Corollary~\ref{NotStrongIff}, this implies that $G$ is a hamiltonian $2$-group: 
	\[ G = Q_8 \times B , \]
for some abelian group~$B$ of exponent $\le 2$. Let $G_2 = \langle i^2, B \rangle$ be the subgroup generated by the elements of order~$2$ in~$G$. 
(Recall that $i^2 = j^2 = k^2$ is the unique element of order~$2$ in~$Q_8$.) This is the centre of~$G$, and it is an abelian group of exponent~$2$ that has index~$4$ in~$G$.
By Lemma~\ref{order2}, we have
    \[ \text{$\varphi(g x) = \varphi(g)\,  \varphi(x)$ 
    \ for all $g \in G$ and $x \in G_2$.} \]
Therefore 
    \begin{align} \label{ColPermOfCompletePf-G2}
    \text{$\varphi(G_2) = G_2$, and the restriction of~$\varphi$ to~$G_2$ is a group automorphism.}
    \end{align}

Note that $\Cay \bigl( G; \{i^{\pm1}\} \bigr)$ is a disjoint union of $4$-cycles (because $i$ has order~$4$). Hence, the isomorphic graph $\Cay \bigl( G; \{\varphi(i)^{\pm1}\} \bigr)$ is also a disjoint union of $4$-cycles, which means that $\varphi(i)$ has order~$4$. Therefore $\varphi(i)^{\pm2} = i^2$, because $g^2 = i^2$ for all $g \in G$ of order 4. Also, we have $\varphi(i^2) \in \{\varphi(i)^{\pm2}\}$ (see Lemma~\ref{phi(s^k)}). Combining these observations yields 
    \[ \varphi(i^2) = i^2 . \]

Now, let $\widetilde Q_8 = \langle \varphi(i), \varphi(j) \rangle$, and define $\alpha \colon Q_8 \to G$ by 
    \[ \alpha(i^m j^n) = \varphi(i)^m \, \varphi(j)^n . \]
(We will soon see that $\alpha$ is well-defined.)
Since $i$ and~$j$ have order~$4$ (and $\varphi \in \caut_\1(K_G)$), we know that $\varphi(i)$ and $\varphi(j)$ have order~$4$ (by the argument in the preceding paragraph). Also note that $i \notin j \, G_2$, so \eqref{ColPermOfCompletePf-G2} implies that $\varphi(i) \notin \varphi(j) \, G_2$; hence, $\varphi(i)$ does not centralize $\varphi(j)$. However, any two elements of~$G$ either commute or invert each other; therefore $\varphi(i)$ and~$\varphi(j)$ invert each other, so $\varphi(i)$ and $\varphi(j)$ satisfy exactly the same relations as~$i$ and~$j$. We conclude that 
    \begin{align} \label{ColPermOfCompletePf-Q8}
    \text{$\alpha$ is a (well-defined) isomorphism from~$Q_8$ to~$\widetilde Q_8$.}
    \end{align}

Also note that, for all $m,n \in \Z$, Lemma~\ref{phi(gh)} implies
    \[ \varphi(i^m j^n)
        \in \{ \varphi(i^m) \, \varphi(j)^{\pm n} \}
        \subseteq \{ \varphi(i)^{\pm m} \, \varphi(j)^{\pm n} \}
        \subseteq \alpha(Q_8)
        = \widetilde Q_8
        ,\]
so $\varphi(Q_8) = \widetilde Q_8$.
Therefore, since $Q_8 \cap G_2 = \langle i^2 \rangle$, we see that
    \[ \widetilde Q_8 \cap \varphi(G_2) 
    = \varphi(Q_8) \cap \varphi(G_2)
    = \varphi(Q_8 \cap G_2) 
    = \varphi \bigl( \langle i^2 \rangle \bigr)
    = \langle i^2 \rangle
    . \]
Since $\varphi(B) \subseteq \varphi(G_2)$ and $i^2 \notin B$, this implies $\widetilde Q_8 \cap \varphi(B) = \{\1\}$. Also, since the restriction of~$\varphi$ to~$G_2$ is a group automorphism, we know that $\varphi(B)$ is a subgroup of~$G$ and we have $|G : \varphi(B)| = |G : B| = 8$; from this, we can conclude that $\widetilde Q_8 \cdot \varphi(B) = G$. So we see from Lemma~\ref{InternalDirProd} that
    \begin{align} \label{ColPermOfCompletePf-DirProd}
    G = \widetilde Q_8 \times \varphi(B)
    . \end{align}

Define $\beta \colon G \to G$ by 
    \[ \text{$\beta(gb) = \alpha(g) \, \varphi(b)$ 
    \ for $g \in Q_8$ and $b \in B$.} \]
We see from \eqref{ColPermOfCompletePf-G2}, \eqref{ColPermOfCompletePf-Q8}, and~\eqref{ColPermOfCompletePf-DirProd} that $\beta$ is an automorphism of~$G$.

Now, let $\psi = \beta^{-1} \compose \varphi$. To complete the proof, we will show that $\psi$ is a colour-preserving automorphism of\/ $\Cay \bigl( G; G \smallsetminus \{\1\} \bigr)$. (This implies that, as desired, $\varphi = \beta \compose \psi$ is the composition of a group automorphism, namely~$\beta$, and a colour-preserving graph automorphism, namely~$\psi$.) First, note that $\psi$ is a colour-permuting automorphism (because it is the composition of two colour-permuting automorphisms) and that $\psi(\1) = \1$ (because $\beta$ and~$\varphi$ both fix~$\1$).

If $m,n \in \Z$, and $n$ is odd, then conjugation by~$j^n$ inverts~$i^m$, so
    \[ i^m j^{-n} 
    = j^{-n} \cdot (j^n i^m j^{-n} )
    = j^{-n} \cdot i^{-m}
    = (i^m j^n)^{-1}
    . \]
On the other hand, if $n$ is even, then $j^{-n} = j^n$, so $i^m j^{-n} = i^n j^n$. Thus, we have 
    \begin{align} \label{ColPermOfComplete-i^mj^n}
    \text{$i^m j^{-n} \in \{ (i^m j^n)^{\pm1}\}$ for all $m,n \in \Z$} 
    . \end{align}

Hence, for all $q \in Q_8$, we have
    \begin{align*}
    \varphi(q)
    &= \varphi(i^m j^n)
        && \text{(for some $m,n \in \Z$)}
    \\&\in \{ \varphi(i^m) \, \varphi(j)^{\pm n} \}
        && \text{(Lemma~\ref{phi(gh)})}
    \\&= \{ \varphi(i)^m \, \varphi(j)^{\pm n} \}
        && \text{(Lemma~\ref{phi(s^k)} with $g = \1$)}  
    \\&= \{ \beta(i)^m \, \beta(j)^{\pm n} \}
        && \text{(definitions of~$\beta$ and~$\alpha$)}
    \\&= \{ \beta(i^m j^{\pm n}) \}
        && \text{($\beta$ is a group homomorphism)}
    \\&\subseteq \bigl\{ \beta \bigl( (i^m j^n)^{\pm 1} \bigr) \bigr\}
        && \text{(\eqref{ColPermOfComplete-i^mj^n}}
    \\&= \{ \beta(q^{\pm 1}) \}
        && \text{(definition of $m$ and~$n$)}
    \\&= \{ \beta(q)^{\pm 1} \}
        && \text{($\beta$ is a group homomorphism)}
    .\end{align*}
Therefore 
    \begin{align} \label{ColPermOfCompletePf-beta(q)}
    \text{$\beta(q) \in \{ \varphi(q)^{\pm 1} \}$  for all $q \in Q_8$} .
    \end{align} 

Now, for any $s \in G$, we have
        \begin{align*} 
        \beta(s) 
        &= \beta(qb) 
            && \text{(for some $q \in Q_8$ and $b \in B$)}
        \\&= \beta(q)\, \beta(b) 
            && \text{($\beta$ is a group automorphism)}
        \\&\in \{ \varphi(q)^{\pm1} \, \beta(b) \}
            && \text{\eqref{ColPermOfCompletePf-beta(q)}}
        \\&= \{ \bigl(\varphi(q) \, \beta(b) \bigr)^{\pm1} \}
            && \begin{pmatrix}
            \text{$\beta(b) \in G_2$ has order~$1$ or~$2$}
            \\ \text{and is in the centre of~$G$}
            \end{pmatrix}
        \\&= \{ \bigl(\varphi(q) \, \varphi(b) \bigr)^{\pm1} \}
            && \text{(definitions of~$\beta$ and~$\alpha$)}
       \\&= \{ \varphi(qb)^{\pm1} \}
            && \text{(Lemma~\ref{order2})}
       \\&= \{ \varphi(s)^{\pm1} \}
            && \text{(definition of $q$ and $b$)}
        . \end{align*}
Since $\varphi(s^{-1}) = \varphi(s)^{-1}$ (by Lemma~\ref{phi(s^k)} with $g = \1$) and $\beta(s^{-1}) = \beta(s)^{-1}$ (because $\beta$ is a group homomorphism), this implies
    \[ \text{$\varphi \bigl( \{s^{\pm1} \} \bigr) = \beta \bigl( \{s^{\pm1} \}\bigr)$ for all $s \in G$} . \]
Therefore, we have
    \[ \psi \bigl( \{s^{\pm1}\} \bigr) 
        = \beta^{-1}\bigl(\varphi \bigl( \{s^{\pm1}\} \bigr) \bigr)
       =  \beta^{-1}\bigl(\beta \bigl( \{s^{\pm1}\} \bigr) \bigr) 
        = \{s^{\pm1}\} 
        . \]
Hence, for all $g,s \in G$, we see from Lemma~\ref{phi(gh)} that 
    \[ \psi(g s) 
    \in \{\psi(g) \, \psi(s^{\pm1}) \} 
    = \{\psi(g) \, s^{\pm1} \} . \]
This means that $\psi$ is colour-preserving, as desired.
\end{proof}

This theorem yields a direct proof of the implication $(\ref{equiv-isCCA} \Rightarrow \ref{equiv-isStrong})$ of Theorem~\ref{equiv}:

\begin{cor} \label{CompleteCCAisStrong}
If the complete graph\/ $\Cay \bigl( G; G \smallsetminus \{\1\} \bigr)$ is CCA, then it is strongly CCA.
\end{cor}

\begin{proof}
Let $\varphi$ be a colour-permuting automorphism of $\Cay \bigl( G; G \smallsetminus \{\1\} \bigr)$. By Theorem~\ref{ColPermOfComplete}, we may write $\varphi = \psi \compose \alpha$, where $\psi$ is a colour-preserving graph automorphism and $\alpha \in \Aut G$. However, since $\Cay \bigl( G; G \smallsetminus \{\1\} \bigr)$ is CCA, we know that $\psi$ is an affine map. Hence, $\varphi$ is the composition of two affine maps, and is therefore affine.
\end{proof}


\begin{thebibliography}{99}

\bibitem{NormalCayleyGraphs}
A.\,S.\,Árnadóttir and C.\,Godsil,
Strongly cospectral vertices in normal Cayley graphs.
\emph{Discrete Math.} \textbf{346} (2023), no.~5, Paper No.~113341, 17~pp.
\MR{4543505}, \url{https://doi.org/10.1016/j.disc.2023.113341}

\bibitem{ByrneDonnerSibley}
D.\,P.\,Byrne, M.\,J.\,Donner, and T.\,Q.\,Sibley,
Groups of graphs of groups.
\emph{Beitr. Algebra Geom.} \textbf{54} (2013), no.~1, 323--332.
\MR{3027684}, \url{https://doi.org/10.1007/s13366-012-0093-7}

\bibitem{DHKM-CCAodd}
E.\,Dobson, A.\,Hujdurović, K.\,Kutnar, and J.\,Morris,
On color-preserving automorphisms of Cayley graphs of odd square-free order.
\emph{J.~Algebraic Combin.} \textbf{45} (2017), no.~2, 407--422.
\MR{3604062},
\url{https://doi.org/10.1007/s10801-016-0711-9}

\bibitem{DummitFoote}
D.\,S.\,Dummit and M.\,Foote, 
\emph{Abstract Algebra}, third ed.,
John Wiley \& Sons, Hoboken, NJ, 2004.
\MR{2286236},
ISBN: 0-471-43334-9

\bibitem{fuller}
B.\,Fuller, Finding CCA groups and graphs algorithmically.
Masters thesis, University of Lethbridge, 2018.
\url{https://hdl.handle.net/10133/4996}

\bibitem{Godsil-GRR}
C.\,D.\,Godsil,
GRRs for nonsolvable groups, in:
\emph{Algebraic Methods in Graph Theory, Vol.~I, II (Szeged, 1978)},
North-Holland Publishing Co., Amsterdam-New York, 1981, pp. 221--239.
\MR{0642043},
ISBN: 0-444-85442-8

\bibitem{Godsil-full}
C.\,D.\,Godsil,
On the full automorphism group of a graph.
\emph{Combinatorica} \textbf{1} (1981), no.~3, 243--256.
\MR{0637829},
\url{https://doi.org/10.1007/BF02579330}

\bibitem{HKMM-ColPres}
A.\,Hujdurovi\'c, K.\,Kutnar, D.\,W.\,Morris, and J.\,Morris, 
On colour-preserving automorphisms of Cayley graphs.
\emph{Ars. Math. Contemp.} \textbf{11} (2016), no.~1, 189–213. 
\MR{3546658},
\url{https://doi.org/10.26493/1855-3974.771.9b3}

\bibitem{LeemandelaSalle-MRR}
P.-H.\,Leemann and M.\,de\,la\,Salle,
Most rigid representation and Cayley index of finitely generated groups.
\emph{Electron. J. Combin.} \textbf{29} (2022), no.~4, Paper No.~4.40, 9 pp.
\MR{4522312},
\url{https://doi.org/10.37236/10512}

\bibitem{Wang}
C.\,Wang, D.\,Wang, and M.\,Xu,
Normal Cayley graphs of finite groups.
\emph{Sci. China Ser.~A} \textbf{41} (1998), no.~3, 242--251.
\MR{1621054},
\url{https://doi.org/10.1007/BF02879042}

\bibitem{Wielandt}
H.\,Wielandt,
\emph{Finite Permutation Groups},
Academic Press, New York-London, 1964. 
\MR{0183775},
ISBN:978-0-12-749650-4

\bibitem{wiki-hamgrp}
Wikipedia, Dedekind group.
\url{https://en.wikipedia.org/wiki/Dedekind_group}

\end{thebibliography}
\end{document}